\documentclass[10pt,twoside]{siamart1116}

\usepackage[english]{babel}
\usepackage{graphicx,epstopdf,epsfig}
\usepackage{amsfonts,epsfig,fancyhdr,graphics, hyperref,amsmath,amssymb,listings}
\usepackage[all]{xy}
\usepackage{xcolor}
\usepackage{tcolorbox}
\usepackage{color, colortbl}

\def\cvd{~\vbox{\hrule\hbox{%
     \vrule height1.3ex\hskip0.8ex\vrule}\hrule } }

\newcommand{\HE}{Namie of Handling Editor}
\newcommand{\DoS}{Month/Day/Year}
\newcommand{\DoA}{Month/Day/Year}

\newcommand{\Names}{Andrii Dmytryshyn}
\newcommand{\Title}{Versal deformations: a tool of linear algebra}

\newtheorem{remark}[theorem]{Remark}
\newtheorem{example}[theorem]{Example}

\newcommand{\reals}{\mathbb{R}}
\newcommand{\complex}{\mathbb{C}}
 
\renewcommand{\theequation}{\thesection.\arabic{equation}}
\renewtheorem{theorem}{Theorem}[section]

\DeclareMathOperator{\codim}{codim}
\DeclareMathOperator{\bun}{\cal B}

\definecolor{codegreen}{rgb}{0,0.6,0}
\definecolor{codegray}{rgb}{0.5,0.5,0.5}
\definecolor{codepurple}{rgb}{0.58,0,0.82}
\definecolor{backcolour}{rgb}{0.95,0.95,0.92}

\lstdefinestyle{mystyle}{
    backgroundcolor=\color{backcolour},   
    commentstyle=\color{codegreen},
    keywordstyle=\color{magenta},
    numberstyle=\tiny\color{codegray},
    stringstyle=\color{codepurple},
    basicstyle=\ttfamily\footnotesize,
    breakatwhitespace=false,         
    breaklines=true,                 
    captionpos=b,                    
    keepspaces=true,                 
    numbers=left,                    
    numbersep=5pt,                  
    showspaces=false,                
    showstringspaces=false,
    showtabs=false,                  
    tabsize=2
}

\begin{document}

\bibliographystyle{plain}

\setcounter{page}{1}

\thispagestyle{empty}

 \title{\Title\thanks{Received
 by the editors on \DoS.
 Accepted for publication on \DoA. 
 Handling Editor: \HE. Corresponding Author: Andrii Dmytryshyn}}

\author{
Andrii Dmytryshyn\thanks{School of Science and Technology,
\"Orebro University (andrii.dmytryshyn@oru.se). The author is supported by the Swedish Research Council (VR) grant 2021-05393.}
}

\markboth{\Names}{\Title}

\maketitle

\begin{abstract}
Versal deformation of a matrix $A$ is a normal form to which all matrices $A+E$, close to $A$, can be reduced by similarity transformation smoothly depending on the entries of $A+E$. In this paper we discuss versal deformations and their use in codimension computations, in investigation of closure relations of orbits and bundles, in studying changes of canonical forms under perturbations, as well as in reduction of unstructured perturbations to structured perturbations.

\end{abstract}





\section{Introduction}
%

To investigate and understand the properties of a matrix we may need to transform this matrix to a simpler form, e.g., Jordan canonical form (JCF) or Weyr canonical form. 
However, the reduction to these forms may be an unstable operation, meaning that both the corresponding canonical form and the reduction transformation can be highly sensitive to small changes in the entries of a given matrix. In the following we recall JCF and illustrate its sensitivity. 

Let $A$ be an $ n\times n$ matrix over the field of complex numbers, denoted by $\mathbb{C}$, and $GL_n(\mathbb{C})$ be a group of $n \times n$ nonsingular complex matrices. Recall that the similarity transformation is defined as follows 
\begin{equation*}
A  \mapsto S^{-1}AS, \text{ where } S \in GL_n(\mathbb{C}).
\end{equation*}
By similarity transformation any matrix can be reduced to its JCF, i.e., for a matrix $A$ there exists a nonsingular matrix $S$, such that $S^{-1}AS = J_{\text{can}}$, where 
\begin{equation}\label{jcf}
\begin{aligned}
J_{\text{can}}&=J(\lambda_1) \oplus J(\lambda_2) \oplus \dots \oplus J(\lambda_n), \\ J(\lambda_i)&=  \bigoplus_{j=1}^{t_i} J_{k_{i,j}}(\lambda_i), \quad \lambda_i \neq \lambda_j, \quad k_{i,1} \geq k_{i,2} \geq \dots \geq k_{i,t_i}, \text{ and }\\
 J_{k_{i,j}}(\lambda_i) &= 
\begin{bmatrix}
\lambda_i&1&&0 \\
& \lambda_i &\ddots& \\
&& \ddots& 1\\
0&&& \lambda_i
\end{bmatrix}, \  \  (\text{the size is } k_{i,j} \times k_{i,j}),
\end{aligned}
\end{equation} 
for more details, see, e.g., \cite{HoJo85}. The JCF of a matrix is well known classical result that has been studied with various purposes. In examples \ref{ex1}  and \ref{ex2}, we compute JCF of some matrices. 
\begin{example}
\label{ex1}
Consider the following perturbation of the diagonal matrix 
\begin{equation*}
A(\varepsilon,\delta)=
\begin{bmatrix}
\lambda&0&0 \\
0& \lambda &0 \\
0&0& \lambda
\end{bmatrix} + 
\begin{bmatrix}
0&\varepsilon&0 \\
0& 0 &\delta \\
0&0& 0
\end{bmatrix}. 
\end{equation*} 
What is JCF of $A(\varepsilon,\delta)$? In the case of $\varepsilon \neq 0$ and $ \delta \neq 0$ (but possibly being very small), a simple computation gives us the following JCF:  
\begin{equation*}
\begin{bmatrix}
1&0&0 \\
0& \varepsilon^{-1}&0 \\
0&0& \varepsilon^{-1}\delta^{-1}
\end{bmatrix}^{-1}
\begin{bmatrix}
\lambda&\varepsilon&0 \\
0& \lambda &\delta \\
0&0& \lambda
\end{bmatrix}
\begin{bmatrix}
1 &0&0 \\
0& \varepsilon^{-1}&0 \\
0&0& \varepsilon^{-1}\delta^{-1}
\end{bmatrix} =
\begin{bmatrix}
\lambda&1&0 \\
0& \lambda &1 \\
0&0& \lambda
\end{bmatrix} = J_3(\lambda).
\end{equation*}
While, in the case of $\varepsilon \neq 0$ and $\delta = 0$, we obtain 
\begin{equation*}
\begin{bmatrix}
1 &0&0 \\
0& \varepsilon^{-1}&0 \\
0&0& 1
\end{bmatrix}^{-1}
\begin{bmatrix}
\lambda&\varepsilon&0 \\
0& \lambda &0 \\
0&0& \lambda
\end{bmatrix}
\begin{bmatrix}
1 &0&0 \\
0& \varepsilon^{-1}&0 \\
0&0& 1
\end{bmatrix} =
\begin{bmatrix}
\lambda&1&0 \\
0& \lambda &0 \\
0&0& \lambda
\end{bmatrix} = J_2(\lambda) \oplus J_1(\lambda).
\end{equation*}
The above shows that the small change in one entry of the matrix, i.e., $\delta$ being arbitralily small but non zero versus $\delta=0$, causes a big effect on the JCF ($J_3(\lambda)$ versus $J_2(\lambda)+J_1(\lambda)$).  
\end{example} 

\begin{example}
\label{ex2}
Consider a perturbation of $J_3(0)$ and compute its JCF using a matrix $S(\varepsilon)$: 
\begin{equation*}
S(\varepsilon)^{-1}\begin{bmatrix}
\lambda&1&0 \\
0& \lambda &1 \\
\varepsilon&0& \lambda
\end{bmatrix}S(\varepsilon) = 
\begin{cases}
\begin{bmatrix}
\lambda_1&0&0 \\
0& \lambda_2 &0 \\
0&0& \lambda_3
\end{bmatrix}, \quad \text{ if } \varepsilon \neq 0,\\
\begin{bmatrix}
\lambda&1&0 \\
0& \lambda &1 \\
0&0& \lambda
\end{bmatrix}, \quad  \text{ if } \varepsilon = 0.
\end{cases} 
\end{equation*} 
In this example the perturbation changes both the sizes of the Jordan blocks and the values of the eigenvalues. 
\end{example}
As can be seen from examples \ref{ex1} and \ref{ex2},  reduction to JCF is an unstable operation: the computed canonical form and the reduction transformation depend discontinuously on the entries of an original matrix. To put it simply, small changes in the entries of an original matrix can cause a big change in its JCF and the transformation that reduces the matrix to this JCF. So reduction of a matrix to its JCF is an ill-posed problem. 
As a consequence of the ill-posedness of reduction to JCF, V.I. Arnold introduced a normal form, with the minimal number of independent parameters, to which an arbitrary family of matrices $\tilde{A}$ close to a given matrix $A$ can be reduced by similarity transformations smoothly depending on the entries of $\tilde{A}$. He called such a normal form a miniversal deformation of $A$ \cite{Arno71,Arno72,Arno97}. 

The rest of the paper is organized as follows. In Section \ref{secdef} we recall the necessary definitions. In Section \ref{seccodim}, we explore the use of miniversal deformations in codimension computations. The characterization of closure relations for bundles of matrices and investigation of possible changes in the JCF (eigenstructures) of matrices in response to small perturbations are discussed in Section \ref{secclo}. Section \ref{secred} provides an algorithm for the reduction of unstructured perturbations to structured perturbations for monic matrix polynomials. 

The main goal of this paper is to present, as simply as possible, various known uses of miniversal deformations, but the paper also contains some new developments.  Namely, our proof of Theorem \ref{defbun}, about the characterization of closure relations for bundles of matrices, has not appeared in the literature before. Neither has Algorithms \ref{alg}, but it solves a problem that is a partial case of already solved problem, see Section \ref{secred} for more details. The advantage of using  Algorithms \ref{alg} (when possible) is that it deals with matrices (rather than matrix pencils) and thus requires less memory and computations.

\section{Versal deformations of JCF} \label{secdef}

In this section we present a formal definition of (mini)versal deformations and a theorem that provides miniversal deformations of JCF. We also present illustrative examples. 

\begin{definition}
A {\it deformation} of a matrix
$A\in{\mathbb C}^{n\times n}$ is a
holomorphic map ${\cal A}\colon
\Lambda\to {\mathbb C}^{n\times n}$ in
which $\Lambda\subset \mathbb C^k$ is a
neighborhood of $\vec 0=(0,\dots,0)$
and ${\cal A}(\vec 0)=A$.
\end{definition}
\noindent On the set of deformations we have an equivalence relation induced by similarity, which is defined as follows.  
\begin{definition}
\label{eqdef}
Deformations ${\cal A}(\vec\varepsilon)$ and
${\cal B}(\vec\varepsilon)$ of a
matrix $A$ are called {\it equivalent} if the
identity matrix $I_n$ possesses a deformation
${\cal I}(\vec\varepsilon) $ such that
\begin{equation*}\label{kft}
\begin{split}
{\cal B}(\vec\varepsilon)=
{\cal I}(\vec\varepsilon)^{-1}
{\cal A}(\vec\varepsilon)
{\cal I}(\vec\varepsilon) 
\end{split}
\end{equation*}
for all $\vec\varepsilon=(\varepsilon_1,\dots, \varepsilon_k)$ in some neighborhood of $\vec 0$.
\end{definition}
\noindent Using the equivalence relation from Definition \ref{eqdef} we are able to find such a deformation that any other deformation can be ``induced'' from it.
\begin{definition}
\label{vdef}
A deformation ${\cal
A}(\delta_1,\dots,\delta_k)$ of a
matrix $A$ is called {\it versal} if
every deformation ${\cal
B}(\vec\varepsilon)={\cal
B}(\varepsilon_1,\dots,\varepsilon_l)$ of $A$ is
equivalent to a deformation of the form
${\cal A}(\varphi_1(\vec\varepsilon),\dots,
\varphi_k(\vec\varepsilon)),$ in which all
$\varphi_i(\vec\varepsilon)$ are power series
that are convergent in a neighborhood
of $\vec 0$ and $\varphi_i(\vec 0)=0$.
\end{definition}
\noindent If a versal deformation ${\cal A}(\delta_1,\dots,\delta_k)$ of a matrix $A$ has a minimal number of independent parameters, i.e., $k$ is minimal, then it is called {\it miniversal}. 

For constructing the miniversal  deformations of JCF we need to define the following $m \times n$
matrices 
\begin{equation*}
0_{mn}^{\leftarrow} = \begin{bmatrix}
\begin{matrix}
*\\ *\\\vdots\\ *
\end{matrix}
&0_{m,n-1}
\end{bmatrix}\quad
 \text{if $m\le n$ \ \  and \ \ }
0_{mn}^{\downarrow} = 
\begin{bmatrix}
\\[1mm]
   0_{m-1,n}
\\[1mm]
  *\ *\ \dots\ *
\end{bmatrix}\quad
\text{if $m\ge n$}, 
 \end{equation*}
where the stars denote all possibly non-zero entries. Further, we will usually omit the indices $m$ and $n$. 
Consider a matrix in the Jordan canonical form $J_{\text{can}}$, see \eqref{jcf}, and define 
\begin{equation}
\label{defJ}
\begin{aligned}
J_{\text{can}} + D&=(J(\lambda_1) + D_1) \oplus (J(\lambda_2)+ D_2) \oplus \dots \oplus (J(\lambda_n)+ D_n), \text{ where } \\ 
J(\lambda_i) + D_i&=  
\begin{bmatrix}
J_{k_{i,1}}(\lambda_i) + 0^{\downarrow}&0^{\downarrow}&0^{\downarrow}& \dots&0^{\downarrow} \\
0^{\leftarrow}&J_{k_{i,2}}(\lambda_i)+ 0^{\downarrow}&0^{\downarrow} &  \dots & 0^{\downarrow} \\
0^{\leftarrow}&0^{\leftarrow}&J_{k_{i,3}}(\lambda_i)+0^{\downarrow}& \dots & 0^{\downarrow} \\
\vdots&\vdots&\vdots&\ddots&\vdots \\
0^{\leftarrow}&0^{\leftarrow}& 0^{\leftarrow} & \dots& J_{k_{i,t_i}}(\lambda_i) + 0^{\downarrow}\\
\end{bmatrix}.
\end{aligned}
\end{equation} 

Theorem \ref{mdarn} below shows that miniversal deformation of a matrix in JCF can be taken in the shape defined in \eqref{defJ}. 

\begin{theorem}[\cite{Arno71,GaSe99,KlSe14}] \label{mdarn}
Let $J_{\text{can}}$ be a matrix in Jordan canonical form. All matrices $J_{\text{can}}+E$ that are sufficiently close to $J_{\text{can}}$ can be reduced by transformations
\begin{equation*}\label{tef}
J_{\text{can}}+E\mapsto {
S}(E)^{-1} (J_{\text{can}}+E) {
S}(E),\quad\begin{matrix}
\text{${S}(E)$
is analytic at 0}\\
\text{and ${S}(0)=I$,}
\end{matrix}
\end{equation*}
to the form $J_{\text{can}} +D(E)$, defined in \eqref{defJ}, 
the non-zero entries of $D(E)$ (i.e., entries at the positions of the stars $*$) depend holomorphically on the entries of $E$.
\end{theorem}
\noindent To illustrate the result of Theorem \ref{mdarn} we consider the following three examples. 
\begin{example}  \label{mindefex} 
Theorem \ref{mdarn} tells us that the perturbed matrix 
\begin{equation*}
\begin{aligned}
&J_{\text{can}} +E = J_3(\lambda)\oplus J_2(\lambda)\oplus J_1(\lambda)\oplus J_2(\mu) +E \\ 
&=
\left( \begin{array}{c|c|c|c}
\begin{matrix}
\lambda&1&0 \\
0& \lambda&1 \\
0&0& \lambda
\end{matrix}& & &\\
\hline
& \begin{matrix}
\lambda&1 \\
0& \lambda
\end{matrix}& &\\
\hline
& & \begin{matrix}
\lambda
\end{matrix} &\\
\hline
& & & \begin{matrix}
\mu&1 \\
0& \mu
\end{matrix} \end{array}
\right)
+
\left(
\begin{array}{ccc|cc|c|cc}
\varepsilon_{11}&\varepsilon_{12}&\varepsilon_{13}&\varepsilon_{14}&\varepsilon_{15}&\varepsilon_{16}&\varepsilon_{17}&\varepsilon_{18} \\
\varepsilon_{21}& \varepsilon_{22} &\varepsilon_{23}&\varepsilon_{24}&\varepsilon_{25}&\varepsilon_{26} &\varepsilon_{27}&\varepsilon_{28}\\
\varepsilon_{31}&\varepsilon_{32}& \varepsilon_{33}&\varepsilon_{34}&\varepsilon_{35}&\varepsilon_{36}&\varepsilon_{37}&\varepsilon_{38}\\
\hline
\varepsilon_{41}&\varepsilon_{42}&\varepsilon_{43}&\varepsilon_{44}&\varepsilon_{45}&\varepsilon_{46} &\varepsilon_{47}&\varepsilon_{48}\\
\varepsilon_{51}& \varepsilon_{52} &\varepsilon_{53}&\varepsilon_{54}&\varepsilon_{55}&\varepsilon_{56}&\varepsilon_{57}&\varepsilon_{58}\\
\hline
\varepsilon_{61}&\varepsilon_{62}& \varepsilon_{63}&\varepsilon_{64}&\varepsilon_{65}&\varepsilon_{66}&\varepsilon_{67}&\varepsilon_{68}\\
\hline
\varepsilon_{71}&\varepsilon_{72}& \varepsilon_{73}&\varepsilon_{74}&\varepsilon_{75}&\varepsilon_{76}&\varepsilon_{77}&\varepsilon_{78}\\
\varepsilon_{81}&\varepsilon_{82}& \varepsilon_{83}&\varepsilon_{84}&\varepsilon_{85}&\varepsilon_{86}&\varepsilon_{87}&\varepsilon_{88}\\
\end{array}
\right)
\end{aligned}
\end{equation*}
can be reduced, by transformation $J_{\text{can}}+E\mapsto {S}(E)^{-1} (J_{\text{can}}+E) {S}(E),$ to the form   
\begin{equation*}
\begin{aligned}
&J_{\text{can}} +D(E) = J_3(\lambda)\oplus J_2(\lambda)\oplus J_1(\lambda)\oplus J_2(\mu) +D(E)\\ 
&=
\left( \begin{array}{c|c|c|c}
\begin{matrix}
\lambda&1&0 \\
0& \lambda&1 \\
0&0& \lambda
\end{matrix}& & &\\
\hline
& \begin{matrix}
\lambda&1 \\
0& \lambda
\end{matrix}& &\\
\hline
& & \begin{matrix}
\lambda
\end{matrix} &\\
\hline
& & & \begin{matrix}
\mu&1 \\
0& \mu
\end{matrix} \end{array}
\right)+ 
\left( \begin{array}{c|c|c|c}
\begin{matrix}
0&0&0 \\
0&0&0 \\
\delta_1 &\delta_2& \delta_3
\end{matrix}& 
\begin{matrix}
0&0\\
0&0 \\ 
\delta_4& \delta_5
\end{matrix}& 
\begin{matrix}
0\\
0\\ 
\delta_6
\end{matrix} 
&\\
\hline
\begin{matrix}
\delta_{7} &0_{\phantom{1}}&0\\
\delta_{8} &0_{\phantom{1}}&0\\ 
\end{matrix}
& 
\begin{matrix}
0&0 \\
\delta_{10} & \delta_{11}
\end{matrix}
&
\begin{matrix}
0\\
\delta_{12} 
\end{matrix}
&\\
\hline
\begin{matrix}
\delta_9 &0_{\phantom{1}}&0\\
\end{matrix}
&
\begin{matrix}
\delta_{13\phantom{1}}&0\\
\end{matrix}
&
\begin{matrix}
\delta_{14}
\end{matrix}
&\\
\hline
\begin{matrix}
&&&&\\
&&&&\\ 
\end{matrix}
& 
\begin{matrix}
& \\
& 
\end{matrix}
&
\begin{matrix}
\\
\end{matrix}
&
\begin{matrix}
0&0 \\
\delta_{15}& \delta_{16}
\end{matrix}\\
\end{array}
\right)
\end{aligned}
\end{equation*}
and $\delta_i=\varphi_i(\vec\varepsilon)$. 
\end{example}

In Example \ref{22md} we consider the case of a perturbed $2 \times 2$ Jordan block. We present explicitly the functions $\delta_i=\varphi_i(\vec\varepsilon)$ in the miniversal deformation of the block and the transformation $S(E)$.  

\begin{example}[Reduction to a miniversal deformation for $2 \times 2$ Jordan block] \label{22md}
We present a reduction of a perturbed $2 \times 2$ Jordan block $J_2(0)  + E$ to its miniversal deformation $J_2(0) + D(E)$. Define: 
$$
J_2(0)  + E = 
\left[ 
\begin{matrix}
0&1\\ 
0&0\\ 
\end{matrix}
\right] + 
\left[ 
\begin{matrix}
\varepsilon_{11}&\varepsilon_{12}\\ 
\varepsilon_{21}&\varepsilon_{22}\\ 
\end{matrix}
\right] \quad \text{and} \quad S(E) = \left[ 
\begin{matrix}
1&0\\
\frac{-\varepsilon_{11}}{1+\varepsilon_{12}}&\frac{1}{1+\varepsilon_{12}}\\ 
\end{matrix}
\right].
$$
Then  
\begin{align*}
J_2(0)  + D(E) = &S(E)^{-1}(J_2(0)  + E)S(E) =\left[ 
\begin{matrix}
1&0\\
\varepsilon_{11}&1+\varepsilon_{12}\\ 
\end{matrix}
\right]
\left[ 
\begin{matrix}
\varepsilon_{11}&1+\varepsilon_{12}\\ 
\varepsilon_{21}&\varepsilon_{22}\\ 
\end{matrix}
\right]
\left[ 
\begin{matrix}
1&0\\
\frac{-\varepsilon_{11}}{1+\varepsilon_{12}}&\frac{1}{1+\varepsilon_{12}}\\ 
\end{matrix}
\right] \\ 
& =\left[ 
\begin{matrix}
0&1\\ 
\varepsilon_{21}(1+\varepsilon_{12}) - \varepsilon_{11} \varepsilon_{22} &\varepsilon_{11} + \varepsilon_{22}\\ 
\end{matrix}
\right]=
\left[ 
\begin{matrix}
0&1\\ 
-{\rm det}(J_2(0)  + E)& {\rm trace}(J_2(0)  + E) \\ 
\end{matrix}
\right].
\end{align*}
In terms of Definition \ref{vdef} we have 
\begin{align*}
{\cal A}(\delta_1,\delta_2)  = {\cal A}(\varphi_1(\vec\varepsilon),&\varphi_2(\vec\varepsilon)) =  \left[ 
\begin{matrix}
0&1\\ 
\varphi_1(\vec\varepsilon)&\varphi_2(\vec\varepsilon)\\ 
\end{matrix}
\right] = \left[ 
\begin{matrix}
0&1\\ 
\varepsilon_{21}(1+\varepsilon_{12}) - \varepsilon_{11} \varepsilon_{22} &\varepsilon_{11} + \varepsilon_{22}\\ 
\end{matrix}
\right],
\end{align*}
i.e., $\delta_1=\varphi_1(\vec\varepsilon) = \varepsilon_{21}(1+\varepsilon_{12}) - \varepsilon_{11} \varepsilon_{22}$  and $\delta_2=\varphi_2(\vec\varepsilon) = \varepsilon_{11} + \varepsilon_{22}$.
\end{example}

It is possible to directly generalize Example \ref{22md} to a single Jordan block of any size, since we can easily see what elementary operations with rows and columns of the matrix are needed to eliminate the unwanted entries. Such a reduction can also be done in Matlab using symbolic computations, see Example \ref{33md}. For the general case (many Jordan blocks), we refer to 
\cite[pp.~5--7]{KlSe14}.

\begin{example}[Reduction to a miniversal deformations for $3 \times 3$ Jordan block]  \label{33md}
In this example we present Matlab code that performs a reduction of a perturbed $3 \times 3$ Jordan block to its miniversal deformation. 
\begin{figure}
\begin{lstlisting}[language=Matlab]
% Creating a fully perturbed 3x3 Jordan block 

syms E [3 3];
J3 = [0 1 0; 0 0 1; 0 0 0];

AE = J3 + E

% Setting the first row to [0 1 0]

S1 = [1 0 0; 
     -AE(1,1)/AE(1,2) 1/AE(1,2) -AE(1,3)/AE(1,2);
      0 0 1];

A1=inv(S1)*(AE)*S1;

% Setting the second row to [0 0 1]

S2 = [1 0 0;
      0 1 0;
    -A1(2,1)/A1(2,3) -A1(2,2)/A1(2,3) 1/A1(2,3)];

A2 = inv(S2)*(A1)*S2;

% Simplyfying the expressions in the third row 

AD = simplify(A2) 
\end{lstlisting}
\caption{Matlab code for reduction to miniversal deformation of perturbed $J_3(0)$.}
\label{codej3}
\end{figure}
Running the code in Figure \ref{codej3} results in the output presented in Figure \ref{rescodej3}. 
\begin{figure}
\begin{lstlisting}[language=Matlab]
AE =
[e1_1, e1_2 + 1,     e1_3]
[e2_1,     e2_2, e2_3 + 1]
[e3_1,     e3_2,     e3_3]

AD =
[ 0, 1, 0]
[ 0, 0, 1]
[e3_1 - e1_1*e3_2 + e1_2*e3_1 - e2_1*e3_3 + e2_3*e3_1 + e1_1*e2_2*e3_3 - e1_1*e2_3*e3_2 - e1_2*e2_1*e3_3 + e1_2*e2_3*e3_1 + e1_3*e2_1*e3_2 - e1_3*e2_2*e3_1, e2_1 + e3_2 - e1_1*e2_2 + e1_2*e2_1 - e1_1*e3_3 + e1_3*e3_1 - e2_2*e3_3 + e2_3*e3_2, e1_1 + e2_2 + e3_3]
\end{lstlisting}
\caption{Perturbed $J_3(0)$ (denoted AE) and its miniversal deformation (denoted AD). The expressions in the last row of AD are the coefficients, with the opposite signs, of the characteristic polynomial of AE.}
\label{rescodej3}
\end{figure}

\end{example}

Beside the miniversal deformations, examples \ref{22md} and \ref{33md}, as well as the reduction process in \cite{KlSe14}, give us an idea of how to construct the transformation matrix, see also algorithms in \cite{DmFS12,DmFS14,Mail99, Mail01} and Section~\ref{secred}. 

 
As one may notice, looking at the definition of  versality (see also the characterization  in Lemma \ref{complemma}, presented in Section \ref{seccodim}), versal or even miniversal deformation of a matrix is not unique. For example, the perturbed matrix $J_{\text{can}}+E$ from Example \ref{mindefex} can also be reduced by similarity transformation to $J_{\text{can}} +\widetilde{D}(E)$,
where
\begin{equation*}
\begin{aligned}
&J_3(\lambda)\oplus J_2(\lambda)\oplus J_1(\lambda)\oplus J_2(\mu) +\widetilde{D}(E)\\ 
&=
\left( \begin{array}{c|c|c|c}
\begin{matrix}
\lambda&1&0 \\
0& \lambda&1 \\
0&0& \lambda
\end{matrix}& & &\\
\hline
& \begin{matrix}
\lambda&1 \\
0& \lambda
\end{matrix}& &\\
\hline
& & \begin{matrix}
\lambda
\end{matrix} &\\
\hline
& & & \begin{matrix}
\mu&1 \\
0& \mu
\end{matrix} \end{array}
\right)+ 
\left( \begin{array}{c|c|c|c}
\begin{matrix}
 \delta_3&0&0 \\
\delta_2& \delta_3&0 \\
\delta_1 &\delta_2& \delta_3
\end{matrix}& 
\begin{matrix}
0&0\\
\delta_5&0 \\ 
\delta_4& \delta_5
\end{matrix}& 
\begin{matrix}
0\\
0\\ 
\delta_6
\end{matrix} 
&\\
\hline
\begin{matrix}
\delta_{7} &0_{\phantom{1}}&0\\
\delta_{8} &\delta_{7}&0\\ 
\end{matrix}
& 
\begin{matrix}
\delta_{11}&0 \\
\delta_{10} & \delta_{11}
\end{matrix}
&
\begin{matrix}
0\\
\delta_{12} 
\end{matrix}
&\\
\hline
\begin{matrix}
\delta_9 &0_{\phantom{1}}&0\\
\end{matrix}
&
\begin{matrix}
\delta_{13\phantom{1}}&0\\
\end{matrix}
&
\begin{matrix}
\delta_{14}
\end{matrix}
&\\
\hline
\begin{matrix}
&&&&\\
&&&&\\ 
\end{matrix}
& 
\begin{matrix}
& \\
& 
\end{matrix}
&
\begin{matrix}
\\
\end{matrix}
&
\begin{matrix}
\delta_{16}&0 \\
\delta_{15}& \delta_{16}
\end{matrix}\\
\end{array}
\right)
\end{aligned}
\end{equation*}
and $\delta_i=\varphi_i(\vec\varepsilon)$. Note, that the number of different functions $\delta_i$ in $\widetilde{D}(E)$ is the same as in $D(E)$ ($D(E)$ is given in Example \ref{mindefex}) but $\widetilde{D}(E)$ has fewer non-zero entries, see more on the miniversal deformations of the shape $\widetilde{D}(E)$ in \cite{Arno71,EdEK97}. 

This paper focuses on versal deformations of JCF of complex matrices. Nevertheless, we also name other known results concerning miniversal deformations and their applications for matrices and matrix pencils in the remarks throughout the paper. 

\begin{remark}[Known deformations for matrices and matrix pencils]
The notion of miniversal deformations has been extended to 
\begin{itemize}
\item matrices under similarity over various fields, see~\cite{BoSS17,Gali72,GaSe99};
\item matrices of bilinear forms ($A \mapsto S^TAS, \text{ det }S \neq 0$), see \cite{DmFS12}; 
\item matrices of sesquilinear forms ($A \mapsto S^HAS, \text{  det }S \neq 0 $), see \cite{DmFS14}; 
\item general matrix pencils under strict equivalence ($A- \lambda B \mapsto RAS- \lambda RBS, \text{  det }S \neq 0, \text{  det }R \neq 0$), see \cite{EdEK97, GaSe99,KlSe10}; 
\item general matrix pencils under contragredient equivalence ($A- \lambda B \mapsto R^{-1}AS- \lambda S^{-1}BR, \text{  det }S \neq 0, \text{  det }R \neq 0$), see \cite{GaSe99}; 
\item structured matrix pencils under congruence ($A- \lambda B \mapsto S^T(A- \lambda B)S, \text{  det }S \neq 0, A^T = \pm A, B^T= \pm B$), see \cite{Dmyt16,Dmyt19,DmFS12,DmFS14}.
\end{itemize}
We also refer the reader to the introductions of \cite{DmFS12,DmFS14} for more information. 
\end{remark}
In the rest of the paper we discuss the use of miniverasal deformations for solving various problems, namely: codimension computations, characterization of closure relations for orbits and bundles, and reduction of the unstructured perturbations to structured perturbations. 


\section{Codimension computation via miniversal deformations} \label{seccodim}
The set of matrices similar to an $n \times n$ matrix $A$ forms a manifold in the complex $n^2$ dimensional space. This manifold is the orbit of $A$ under the action of similarity:
\begin{align*}
{\cal O}(A)= \{ C^{-1}AC: C \in GL_n(\mathbb{C}) \}.
\end{align*}
The vector space
\begin{align*}\label{tansp}
T(A):=\{XA-AX:
X\in{\mathbb
C}^{n\times n}\} 
\end{align*}
is the tangent space
to the similarity orbit of
$A$ at the point
$A$ since 
\begin{multline*}
(I-\varepsilon X)^{-1}A (I-\varepsilon X) = (I+\varepsilon X+\varepsilon^2 X^2+\varepsilon^3 X^3+ \dots) A (I-\varepsilon X)\\
 =  \ A \ + \ 
\underbrace{\varepsilon(XA - AX)}_{\text{order 1 in $\varepsilon$}} \  + \ 
\underbrace{\varepsilon^2 X(I-\varepsilon X)^{-1}(XA-AX)}_{\text{order 2 in $\varepsilon$}} \,
\end{multline*}
for all $n$-by-$n$ matrices $X$ and each $\varepsilon\in\mathbb C$.
Lemma \ref{complemma} shows that the tangent space plays  an important role in the characterization of versal deformations. 
\begin{lemma}[Section 2.3 in \cite{Arno71}]
\label{complemma}
Let $A$ and $E$ be $n \times n$ matrices. Then
  $$A+{D}(E) \text{ is versal if and only if } \ {\mathbb C}^{\,n \times n} = T(A) + {D}({\mathbb C}),$$ where $T(A)$ is the tangent space to $ {\cal O}(A)$, at the point $A$ and  ${D}({\mathbb C})$ is a space of matrices of the form ${D}(E)$, where the functions $\varphi_i(\vec{\varepsilon})$ are replaced by complex numbers.
\end{lemma}

The {\it dimension of the orbit} of $A$ is the dimension of its tangent space at the point $A$.
The {\it codimension of the orbit} $A$ is the dimension of the normal space of its orbit at the point
$A$  which is equal to $n^2$ minus the dimension of the orbit. Lemma \ref{complemma} implies that the codimension of the orbit of $A$ is equal to the minimal possible dimension of the space ${\cal D}(\mathbb C)$ and the latter is also equal to the minimal number of independent parameters in the matrices from ${\cal D}(\mathbb C)$. Therefore miniversal deformations automatically provide us the codimensions of orbits. Summing up, 
\begin{equation*}\label{syst}
\begin{split}
\codim ( &\mathcal{O} (A)) 
= \text{ \# \{functions $\varphi_i(\vec{\varepsilon})$ in the miniversal deformation of $A$\},}
\end{split}
\end{equation*}
where $ \# \Omega$ is the number of elements in the set $\Omega$. 
Note that the codimension of the orbit of $A$ is also equal to the number of linearly independent solutions of the matrix equation
$XA-AX=0$, 
for more details, see e.g., \cite{DeEd95}.

A bundle $\mathcal{B} (A)$ is a union of matrix orbits with the same Jordan structures except that the distinct eigenvalues may be different. Bundles appear naturally in various applications and have been studied extensively, see, e.g., \cite{DeDo23,DFKK15,EdEK99}. Since the eigenvalues in bundles may vary, they become additional parameters resulting in the following codimension formula: 
\begin{equation}\label{buncod}
\begin{split}
\codim ( \mathcal{B} &(A))  = \codim ( \mathcal{O} (A)) - \#\text{ \{different  eigenvalues of $A$\} }\\
&= \text{\# \{functions $\varphi_i(\vec{\varepsilon})$ in the miniversal deformation of $A$\}} \\
&- \text{\# \{different  eigenvalues of $A$\}}.
\end{split}
\end{equation}
From \eqref{buncod} we can conclude that the blocks corresponding to eigenvalues of algebraic multiplicity one, contribute nothing to the codimension of the corresponding bundle. Example \ref{cobumd} illustrates this and shows how easy it can be to compute the codimensions via miniversal deformations.  
\begin{example} \label{cobumd}
Let $A_q$ be a matrix in JCF that has one $J_3(\lambda_1)$ block and $q-1$ $J_1(\lambda_k), k = 2,\dots,q$ blocks. Miniversal defrmation of $A_q$ is 
$$
A_q +D_q= \left[\begin{array}{c}
\begin{matrix}
\lambda_1 &1&0 \\
0& \lambda_1&1 \\
\delta_1&\delta_2& \lambda_1+\delta_3
\end{matrix}
\end{array}
\right]
\oplus  [\lambda_2+\delta_4] \oplus \cdots \oplus [\lambda_q+\delta_{q+2}].
$$
\noindent By counting the number of $\delta_i$ in $A_q+D_q$ (it is equal to ${q+2}$) and subtracting from it the number of  different eigenvalues that $A_q+D_q$ has (it is equal to $q$), we obtain that $\codim ( \mathcal{B} (A_q))= 2$ for any value of $q$. 
Note that the dimensions of bundles $\mathcal{B} (A_q)$ are different for different $q$. 
\end{example}
It is possible to compute the codimension of the similarity orbit of a matrix using the Matrix Canonical Structure Toolbox for Matlab (MCS Toolbox) \cite{DmJK13}. The toolbox was created to simplify working with canonical forms and it can compute  codimensions for various cases, see Remark \ref{remMCS} for more details.  

\begin{remark}[MCS Toolbox] 
\label{remMCS}
MCS Toolbox's functionality includes computation of the codimensions of orbits and bundles for 
\begin{itemize}
\item matrices under congruence, and *congruence, for the theoretical results, see, e.g., \cite{DeDo11b,DeDo11a,DmFS12,DmFS14};    
\item matrix pencils under strict equivalence \cite{DeEd95};
\item controllability and observability matrix pairs \cite{ElJK09};
\item symmetric matrix pencils under congruence \cite{Dmyt19,DmKS14};  
\item skew-symmetric matrix pencils under congruence \cite{Dmyt17,DmKS13}. 
\end{itemize}
\end{remark}

\begin{tcolorbox}[colframe=blue!15!white,colback=blue!5!white]
{\bf Summary of codimension computations.} 
The number of functions $\varphi_i(\vec{\varepsilon})$ in a miniversal deformation of a matrix is equal to the codimension of the similarity orbit of this matrix. 
Therefore computing a miniversal deformation automatically provides us with the codimension.
\end{tcolorbox}

\section{Closure relation of orbits and bundles} \label{secclo}

The problem of changes of JCF under arbitrarily small perturbations is equivalent to the problem of describing what similarity orbits are in the closure of a given similarity orbit. More precisely,  $\mathcal{O} (A) \subset \overline{\mathcal{O} (B)}$ is equivalent to the fact that for every positive $\varepsilon > 0$ there is a perturbation $E, \ \| E \|_F < \varepsilon,$ and a  nonsingular matrix $S$ such that $S^{-1}(A+E)S=B$. Recall that $\| \cdot \|_F$ denotes the Frobenius norm of a matrix. 
If we allow the values of the distinct eigenvalues of $A+E$ to vary, i.e., $S^{-1}(A+E)S=B_E$ and $B_E \in \mathcal{B} (B)$, then the problem of changes of JCF (up to the values of the distinct eigenvalues) of $A+E$ is equivalent to deciding whether the bundle $\mathcal{B} (A)$ is in $\overline{\mathcal{B} (B)}$. Notice that we perturb only a single matrix $A$ but make a conclusion about the whole $\mathcal{O} (A)$ or $\mathcal{B} (A)$ belonging to $\overline{\mathcal{O} (B)}$ or $\overline{\mathcal{B} (B)}$, respectively. For the orbits this conclusion is indeed straightforward but for the bundles it requires some explanations. In \cite{DeDo23} Dopico and De Ter\'{a}n considered such a question about closure relations for bundles of matrix pencils. Moreover, their arguments may  be applied  to  the case of closure relations for bundles of matrices, see \cite[Section 4.1]{DeDo23}. In Theorem \ref{defbun} we provide a new proof of the latter case using miniversal deformations. 

\begin{theorem} \label{defbun}
Let $A \in \bun(B)$ and $B \in \overline{\bun(C)}$ then $A \in \overline{\bun(C)}$.
\end{theorem}
\begin{proof}
Without a loss of generality we may assume that $B$ is in its JCF: 
\begin{equation}
\label{Bjcf}
B = J(\lambda_1) \oplus J(\lambda_2) \oplus \dots \oplus J(\lambda_n), \ \text{ where } \ J(\lambda_i) =  \bigoplus_{j} J_{k_{i,j}}(\lambda_i), \quad \lambda_i \neq \lambda_j. 
\end{equation}
For $A$ there is a non-singular matrix $P$ such that 
\begin{equation}
\label{Ajcf}
P^{-1}AP = J(\mu_1) \oplus J(\mu_2) \oplus \dots \oplus J(\mu_n), \ \text{ where } \ J(\mu_i) =  \bigoplus_{j} J_{k_{i,j}}(\mu_i), \quad \mu_i \neq \mu_j. 
\end{equation}
We emphasize, that the values $n$ and $k_{i,j}$ for all $i,j$ in \eqref{Bjcf} and \eqref{Ajcf} are the same, since $A \in \bun(B)$ (i.e., $A$  and $B$ have the same JCF up to the values of eigenvalues). 

Since $B \in \overline{\bun(C)}$  then there is a perturbation $D$ of $B$, such that $B+D \in \bun(C)$ and $D$ is in the shape of miniversal deformation \eqref{defJ}, i.e., 
$$
B +D = (J(\lambda_1)+D_1) \oplus (J(\lambda_2)+D_2) \oplus \dots \oplus (J(\lambda_n)+D_n). 
$$
The JCF of $B+D$ is a direct sum of the JCF of direct summands $J(\lambda_i)+D_i$. Therefore let $S = \oplus_{i=1}^{n} S_i$ be a non-singular matrix that reduces $B+D$ to  its JCF, i.e., 
\begin{equation}
\label{BPjcf}
\begin{aligned}
S^{-1}(B+D)S &= S^{-1} \big( (J(\lambda_1)+D_1) \oplus (J(\lambda_2)+D_2) \oplus \dots \oplus (J(\lambda_n)+D_n) \big)S \\
&= \bigoplus_{i=1}^n S^{-1}_i (J(\lambda_i)+D_i )S_i= \bigoplus_{i=1}^n \big(S^{-1}_i(J(0)+D_i)S_i + \lambda_i I \big). 
\end{aligned}
\end{equation}
Now consider the perturbation $D$ of $P^{-1}AP$ and apply the similarity transformation with the matrix $S$ to it: 
\begin{equation}
\label{APjcf}
\begin{aligned}
S^{-1}(P^{-1}AP +D)S 
&= \bigoplus_{i=1}^n S^{-1}_i (J(\mu_i)+D_i )S_i= \bigoplus_{i=1}^n \big(S^{-1}_i(J(0)+D_i)S_i + \mu_i I \big). 
\end{aligned}
\end{equation}
Note that since $D$ can be chosen with arbitrarily small entries, we can assume that the direct summands $S^{-1}_i(J(0)+D_i)S_i + \lambda_i I$ and $S^{-1}_j(J(0)+D_j)S_j + \lambda_j I$ have no eigenvalues in common for $i \neq j$  as  well as that the direct summands $S^{-1}_i(J(0)+D_i)S_i + \mu_i I$ and $S^{-1}_j(J(0)+D_j)S_j + \mu_j I$ have no eigenvalues in common for $i \neq j$. Therefore, \eqref{BPjcf} and \eqref{APjcf} show the JCFs of $B+D$ and $P^{-1}AP +D$ only differ in the values of eigenvalues (the eigenvalues of each $S^{-1}_i(J(0)+D_i)S_i$ are shifted with $\lambda_i$ and $\mu_i$, respectively). Thus $P^{-1}AP +D \in \bun(C)$ and, since $D$ is arbitrarily small, we have $P^{-1}AP \in \overline{\bun(C)}$. Therefore  $A \in \overline{\bun(C)}$. 
\end{proof}


For the necessary and sufficient conditions for a matrix $A$ (or, by Theorem \ref{defbun}, a bundle $\mathcal{B} (A)$) being in the closure of another bundle, see, e.g., \cite{EdEK97,Kraf84,KrPr82}. The closure hierarchies of bundles (and also orbits) can be represented as graphs, so called stratification graphs \cite{Dmyt17,DmJK17,DJKV20,DmKa14,EdEK99}. Construction of such graphs is a way to study qualitatively how small perturbations can change the JCF of a matrix and find out what JCF matrices may have in an arbitrarily small neighbourhood of a given matrix. Miniversal deformations may simplify construction of such graphs since they allow us to take into account all possible perturbations of a matrix while working with the matrix, where only a few entries are perturbed, see Example~\ref{per32}.   
Note also that codimensions (discussed in Section \ref{seccodim}) play an important role in investigation of closure relation for orbits and bundles due to the fact that a given orbit (or bundle) has only orbits (or bundles) with higher codimensions in its closure. Thus codimension count provides us with a necessary but not sufficient condition for one orbit (or bundle) being in the closure of another orbit (or bundle). 

%
%
\begin{example} \label{per32}
To investigate what JCFs matrices in an arbitrarily small neighbourhood of a matrix $J_3(\lambda) \oplus  J_2(\lambda)$ may have, i.e., what JCFs $J_3(\lambda) \oplus  J_2(\lambda) + E$ (25 parameters, that are the entries of $E$) may have, it is enough to  investigate what JCFs $J_3(\lambda) \oplus  J_2(\lambda) + D(E)$ (9 parameters, see the matrix to the left in Figure~\ref{fig1}) may have. For example, matrices with the JCFs $J_5(\lambda), J_4(\lambda) \oplus  J_1(\lambda),$ and $J_1(\lambda_1) \oplus J_1(\lambda_2) \oplus J_1(\lambda_3) \oplus  J_2(\lambda)$ are in an arbitrarily small neighbourhood of a matrix $J_3(\lambda) \oplus  J_2(\lambda)$, see Figure \ref{fig1}. 
\begin{figure}
\begin{equation*}
\left[ \begin{array}{c|c}
\begin{matrix}
\lambda&1& \\
& \lambda&1 \\
\delta_1 & \delta_2 & \lambda + \delta_3
\end{matrix}& \begin{matrix}
&\\
& \\ 
\delta_4 && \delta_5
\end{matrix} \\
\hline
\begin{matrix}
\delta_6&&&&&&\\
\delta_7&&&&&&\\ 
\end{matrix}
& 
\begin{matrix}
\lambda&1 \\
\delta_{8} & \lambda + \delta_{9}
\end{matrix}
\end{array}
\right] \ \ \sim \ \
\begin{cases}
\left[ \begin{array}{c}
\begin{matrix}
\lambda&1&&& \\
& \lambda&1&& \\
&& \lambda&1&\\
&& &\lambda&1\\
&& &&\lambda\\
\end{matrix}
\end{array}
\right]  \quad &\text{if only } \delta_4 \neq 0,  \\
\left[ \begin{array}{c|c}
\begin{matrix}
\lambda&1&& \\
& \lambda&1& \\
&& \lambda&1\\
&& &\lambda\\
\end{matrix}& \\
\hline
& \lambda
\end{array}
\right]  \quad &\text{if only } \delta_5 \neq 0,\\
\left[ \begin{array}{c|c}
\begin{matrix}
\lambda_1&& \\
& \lambda_2& \\
&& \lambda_3
\end{matrix}& \\
\hline
& \begin{matrix}
\lambda&1 \\
& \lambda
\end{matrix}
\end{array}
\right]  \quad &\text{if only } \delta_1 \neq 0.\\
\end{cases}
\end{equation*}
\caption{We state what $\vec{\delta}=(\delta_1, \dots , \delta_9)$ can be chosen to show that $J_5(\lambda), J_4(\lambda) \oplus  J_1(\lambda),$ and $J_1(\lambda_1) \oplus J_1(\lambda_2) \oplus J_1(\lambda_3) \oplus  J_2(\lambda)$ are in an arbitrarily small neighbourhood of a matrix $J_3(\lambda) \oplus  J_2(\lambda)$. ($A \sim B$ means that $S^{-1}A(\vec{\delta})S= B$ for some $\vec{\delta}$ and nonsingular $S$.)}
\label{fig1}
\end{figure}
\end{example}
In Remark \ref{remm} we provide a few examples of how various miniversal deformations were used for showing which canonical forms matrices and matrix pencils, have or can not have in an arbitrarily small neighbourhood of a given matrix or matrix pencil. 


\begin{remark}[Changes of canonical forms under perturbations] \label{remm}
Miniversal deformations can be used for studying changes of other (than JCF) canonical forms under arbitrarily small perturbations. Here we list a few examples of such usage: 
\begin{itemize}
\item In \cite[Example 2.1]{Dmyt16}, miniversal deformations of skew-symmetric matrix pencils are used to  show that in an arbitrarily small neighbourhood of a matrix pencil with the canonical form ${\cal L}_1\oplus {\cal L}_0$ there is always a matrix pencil with the canonical form ${\cal H}_2 (\lambda), \lambda \neq 0$, for the definitions of canonical blocks ${\cal L}_k$ and ${\cal H}_n (\lambda)$ and more details, see \cite{Dmyt16}. 
\item In \cite[Theorem 2.3]{DFKK15} miniversal deformations of matrices of bilinear form are used to show that there is a small neighbourhood of a matrix with the canonical form 
{\tiny $\begin{bmatrix}
0&1&0 \\
-1& 0 &0 \\
0&0&0
\end{bmatrix}$}
 that does not contain a matrix with the canonical form 
{\tiny $\begin{bmatrix}
0&1&0 \\
\lambda &0 &0\\
0&0& 0
\end{bmatrix}$}.
\item In \cite{FKKS21} miniversal deformations of matrix pencils are used to calculate the Kronecker canonical form of pencils that are close to any given matrix pencil, i.e., the authors develop a qualitative perturbation theory of matrix pencils through miniversal deformations. 
\end{itemize}
\end{remark}

\begin{tcolorbox}[colframe=blue!15!white,colback=blue!5!white]
{\bf Summary of closure relations of orbits and bundles.}
Miniversal deformations can be used for studying when a closure of  an orbit (or a bundle) contains another orbit (or bundle). Such a closure relation  for orbits or bundles corresponds to changes in the canonical forms (eigenvalues, their multiplicities, minimal indices) of matrices 
under arbitrarily small perturbations. 
\end{tcolorbox}

\section{Reduction to structured perturbations.}\label{secred}

(Mini)versal deformations have or may be forced to have a certain structure, e.g., blocking, sparcity. Therefore the theory of versal deformations provides a possibility to take into account all the possible perturbations of a given matrix while working only with their versal deformations, i.e., only with particularly structured matrices. This property of versal deformations may be used in various ways. 
In particular, in this section, we show how to reduce a perturbation of monic matrix polynomial linearization to a linearization of the perturbed polynomial, or in the other words, how to find which perturbations of the matrix coefficients of a monic matrix polynomial correspond to a given perturbation of the entire linearization. Perturbed polynomial must  remain monic, i.e., the identity matrix in front of $\lambda^d$ is not perturbed. We also derive the transformation matrix that, via similarity, transforms perturbation of the linearization to the linearization of perturbed polynomial. 
The described reduction is possible since linearization of a perturbed polynomial is a versal deformation for perturbation of matrix polynomial linearization \cite{DJKV20,DLPV18,VaDe83}. 


Let $P(\lambda) = \lambda^d + A_{d-1}\lambda^{d-1} + \ldots + A_1\lambda + A_0$, where $A_i \in \mathbb C^{n \times n}, i= 0, \ldots d-1,$ be a matrix polynomial. To compute the eigenvalues of $P(\lambda)$ it is enough to compute the eigenvalues of 
\begin{equation}
\label{comp1}
C_{P}= 
\begin{bmatrix}
-A_{d-1}&-A_{d-2}&\dots&-A_{0}\\
I_n&0&\dots&0\\
&\ddots&\ddots&\vdots\\
0&&I_n&0\\
\end{bmatrix}, 
\end{equation}
since this matrix is a linearization of the polynomial $P(\lambda)$, see e.g., \cite{MaMT15} (note also that \eqref{comp1} is similar to the first companion matrix of $P(\lambda)$ \cite[p.13, Theorem~1.1]{LaRo09}). 
We define a full perturbation of $C_{P}$ as follows 
$$
C_{P}+E=C_{P} + 
\begin{bmatrix}
E_{11}&E_{12}&\dots&E_{1d}\\
E_{21}&E_{22}&\dots&E_{2d}\\
\vdots&\vdots&\ddots&\vdots\\
E_{d1}&E_{d2}&\dots&E_{dd}\\
\end{bmatrix}.
$$
Notably, the full perturbation does not preserve the block-structure of $C_{P}$. 
Therefore we also define the structured perturbation 
$$
C_{P+F(E)}=C_{P} + 
\begin{bmatrix}
F_{d-1}&F_{d-2}&\dots&F_{0}\\
0&0&\dots&0\\
\vdots&\vdots&&\vdots\\
0&0&\dots&0\\
\end{bmatrix}. 
$$
$C_{P+F(E)}$ preserves the block-structure of $C_{P}$ and perturbs only the blocks that correspond to the matrix coefficients of the matrix polynomial $P(\lambda)$. As mentioned before, this structured perturbation is actually a versal (but not miniversal) deformation of $C_{P}$ and thus it is always possible to find a nonsingular matrix $S:=S(E)$, such that $S^{-1} \cdot \left(C_{P}+E \right) \cdot S = C_{P+F(E)}$, see e.g., \cite{DmDo17,DJKV20,DLPV18,VaDe83}. 
Below we present an algorithm for finding this structured perturbation. 

Define a split of a matrix $M=[M_{ij}]$ into a sum of its structured and unstructured parts, ${M}^s$ and ${M}^u$, respectively, as follows:   
\begin{multline*}
\begin{bmatrix}
{  M_{11}}&{  M_{12}}&{   \dots}&{  M_{1d}}\\
{  M_{21}}&{  M_{22}}&{  \dots}&{  M_{2d}}\\
{  \vdots}&{  \vdots}&{  \ddots}&{  \vdots}\\
{  M_{d1}}&{  M_{d2}}&{  \dots}&{  M_{dd}}\\
\end{bmatrix} = 
\begin{bmatrix}
{  M_{11}}&{  M_{12}}&{   \dots}&{  M_{1d}}\\
0&0&{\dots}&0\\
{\vdots}&{\vdots}&{\ddots}&{\vdots}\\
0&0&{\dots}&0\\
\end{bmatrix} +
\begin{bmatrix}
0&0&{ \dots}&0\\
{  M_{21}}&{  M_{22}}&{  \dots}&{  M_{2d}}\\
{  \vdots}&{  \vdots}&{  \ddots}&{  \vdots}\\
{  M_{d1}}&{  M_{d2}}&{  \dots}&{  M_{dd}}\\
\end{bmatrix},\\
\text{and } \ \ { {M}^s} := 
\begin{bmatrix}
{  M_{11}}&{  M_{12}}&{   \dots}&{  M_{1d}}\\
0&0&{\dots}&0\\
{\vdots}&{\vdots}&{\ddots}&{\vdots}\\
0&0&{\dots}&0\\
\end{bmatrix}, \quad 
{  {M}^u} :=
\begin{bmatrix}
0&0&{ \dots}&0\\
{  M_{21}}&{  M_{22}}&{  \dots}&{  M_{2d}}\\
{  \vdots}&{  \vdots}&{  \ddots}&{  \vdots}\\
{  M_{d1}}&{  M_{d2}}&{  \dots}&{  M_{dd}}\\
\end{bmatrix}. 
\end{multline*}
The idea of the algorithm is based on performing similarity transformations on $C_{P} + E$ that decreases the norm of the unstructured part of the perturbation, i.e., the norm of $E^u$:
\begin{align*}
(I-X)^{-1} &(C_{P} + E) (I-X)= (I+X+X^2+X^3+ \dots ) (C_{P} + E) (I-X) \\ 
&= C_{P} + (E + X(C_{P} + E^s) - (C_{P} + E^s)X)^s \\
&\phantom{C_{P} + (E++ (E++ (} + \underbrace{ (E + X(C_{P} + E^s) - (C_{P} + E^s)X)^u}_{\text{we set it to zero, i.e., we find $X$ that eliminates $E^u$;}} +  
\underbrace{X E^u - E^u X + O(X^2).}_{\text{spoils the unstructured part again;}} \,
\end{align*}
Note that the norm of the unstructured part actually decreases since the norm of $X$ is small ($(I-X)$ is a small perturbation of $I$ and the  norm of $E$ is small).  We repeat the above procedure until the norm of the unstructured part becomes sufficiently small. This is formalized in Algorithm~\ref{alg} below. 
\begin{algorithm}
\caption{(Recovering a perturbation of a monic matrix polynomial from a perturbation of its linearization)}
\label{alg}
Let ${  C}_{P}$ be a linearization of a monic matrix polynomial $P(\lambda)$ and ${E}_1$ be a full perturbation of ${  C}_{P}$. Let also $I$ be the identity matrix. 
\begin{itemize}
\item[] {Input:} Monic matrix polynomial $P(\lambda)$, perturbed matrix $ C_{P}+{ E}_1$, and the tolerance parameter ${\rm tol}$;
\vskip0.2cm
\item[] {Initialization:} $S:=I$; 
\vskip0.2cm
\item[] {Computation:} While $\| {E^u_i} \|_F > {\rm tol}$ 
\begin{itemize}
\vskip0.2cm
\item {find the minimum norm least-squares solution to} the Sylvester matrix equation:  

$\big(X_i( C_{P} + E_{i}^{s}) - ( C_{P} + E_{i}^{s})X_i \big)^u=  - {E}^u_i$;
\vskip0.2cm
\item by solving a system of linear equations with multiple right-hand sides, extract the new perturbation ${E}_{i+1}$: 
$(I-X_i)E_{i+1}= E_{i}(I-X_i) - C_{P}X_i + X_iC_{P}$;
\vskip0.2cm
\item construct the new perturbation $C_{P} + {E}_{i+1} $ of the matrix $C_{P}$ {(note that the perturbed linearization $C_{P} + E_{i+1}$ remains similar to the original one $C_{P} + E_{1}$)}; 
\vskip0.2cm
\item update the transformation matrix: $S_{i+1} := S_{i}(I-X_i)$; 
\vskip0.2cm
\item increase the counter: \ $i:=i+1$;
\end{itemize}
\vskip0.2cm
\item[] {Output:} { A structurally} perturbed linearization $ C_{P+F(E)}:= C_{P} + {E}_{k}$, where ${E}_{k}$ is a structured perturbation ({since $\|{ E}^u_{k}\|_F < {\rm tol}$}), and the transformation matrix $S$. \\
\end{itemize}
\end{algorithm}

In \cite{Dmyt22} an algorithm that performs such a reduction for general (possibly non-monic) matrix polynomials and their first companion linearization is presented. Therefore, we refer the interested readers to \cite{Dmyt22} for more details and analysis of such algorithms.  

%
%
We note also that the {construction of the} transformation matrices in Algorithm~\ref{alg}  is similar to the construction of the transformation matrices for the reduction to miniversal deformations of matrices in \cite{DmFS12,DmFS14}. 
\vskip0.3cm
\begin{tcolorbox}[colframe=blue!15!white,colback=blue!5!white]
{\bf Reduction to structured perturbations.} 
The theory of versal deformations provides a possibility to take into account all the possible perturbations of a given matrix while working only with its versal deformations. And since versal deformations have or may be forced to have a certain structure, e.g., blocking, sparcity, then we only need to investigate particularly structured matrices.
\end{tcolorbox}

\section*{Acknowledgements}


This research was supported by the Swedish Research Council (VR) grant 2021-05393. The author is thankful to Froil\'an Dopico for the encouragement to write this paper. 

\footnotesize{
\bibliographystyle{abbrv}
\bibliography{library}
}
\end{document}